\documentclass[a4paper]{amsart}
\usepackage{amsmath,amssymb, amsfonts}
\usepackage{graphicx}
\usepackage{tikz}
\theoremstyle{plain}
 \newtheorem{thm}{Theorem}[section]
 \newtheorem{theorem}{Theorem}[section]      % new
 
 \newtheorem{lemma}[thm]{Lemma}              % lem
 \newtheorem{proposition}[thm]{Proposition}  % prop
\theoremstyle{definition}
 \newtheorem{definition}[thm]{Definition}    % defn
 \newtheorem{exmp}[thm]{Example}

\theoremstyle{remark}
 \newtheorem{remark}[thm]{Remark}            % rem

\usepackage{graphicx, color}
\usepackage{tikz}
\numberwithin{equation}{section}
\DeclareMathOperator{\mc}{mc}
\DeclareMathOperator{\codim}{codim}

\begin{document}
\title{Stable hyperplane arrangements}

\author{Toshio Oshima}% Author Name (\sc should NOT be used here)
\email{oshima@ms.u-tokyo.ac.jp}

% Subject code(s)

\maketitle
\begin{quote}
{\small%
{\bf Abstract.} \ 
We classify complex hyperplane arrangements $\mathcal A$ whose intersection posets $L(\mathcal A)$ satisfy $L(\mathcal A)=\pi_i^{-1}\circ\pi_i\bigl(L(\mathcal A)\bigr)$ for $i=1,\dots,n$. 
Here $\pi_i$ denotes the projection from $\mathbb C^n$ onto $\mathbb C^{n-1}$ defined by that forgets the coordinate $x_i$ of $(x_1,\dots,x_n)\in\mathbb C^n$, and $\pi_i\bigl(L(\mathcal A)\bigr)=\{\pi_i(S)\mid S\in L(\mathcal A)\}$.
We show that such arrangements $\mathcal A$ arise as pullbacks of the mirror hyperplanes of complex reflection groups of type $A$ or $B$.
}
\end{quote}
%%%%%%%%%%%%%%%%%%%%%
%% 1. Introduction %%
%%%%%%%%%%%%%%%%%%%%%
\section{Introduction}
Let $\mathcal A$ be a hyperplane arrangement in $V=\mathbb C^n$.
That is, $\mathcal A$ is a finite union of hyperplanes :
\[
 A\ni H :=\{f_H(x)=0 \mid x\in V\}\text{ where each }f_H(x)\text{ is a polynomial of degree 1.}
\]
If $0\in H$ for every $H\in \mathcal A$, we say that $\mathcal A$ is {\em homogeneous}.	
%%%%%%%%%%%%%%%%
\begin{definition}  We denote by $L(\mathcal A)$, or simply by $\mathcal L$, the set or poset
of affine subspaces of $V$ obtained as intersections of hyperplanes in $\mathcal A$
%%%%

\vspace{2.8cm}
\hspace{9,45cm}\scalebox{0.7}{\begin{tikzpicture}
\draw
(-1.5,0)--(1.5,0) (0,-1.5)--(0,1.5)
(-0.3,-1.5)--(1.5,1.5)  (-1.5,0.7)--(1.5,-1.5)
;
\node at (0.6,0) {$\bullet$};
\node at (0,-0.4) {$\bullet$};
\node at (0,0) {$\bullet$};
\node at (-0.55,0) {$\bullet$};
\node at (0.25,-0.59) {$\bullet$};
\end{tikzpicture}}

\vspace{-5.3cm}
\begin{align*}
\mathcal L=L(\mathcal A)&:=\Bigl\{\bigcap\limits_{H\in\mathcal B}H\ne\emptyset\mid \mathcal B\subset\mathcal A\bigr\}
\end{align*}
and put
\begin{align*}
\mathcal L^{(k)}=L(\mathcal A)^{(k)}&:=\{S\in L(\mathcal A)\mid \codim S=k\},\quad
\mathcal L^{(0)}=\{V\},\quad \mathcal L^{(1)}=\mathcal A.
\intertext{For affine subspaces $S,\,S'\in\mathcal L$, we define}
\mathcal L^{(k)}_{\subset S}&=L^{(k)}_{\subset S}(\mathcal A):=\{T\in\mathcal L^{(k)}\mid T\subset S\},
\allowdisplaybreaks\\
\mathcal L^{(k)}_{\supset S}&=L^{(k)}_{\supset S}(\mathcal A):=\{T\in\mathcal L^{(k)}\mid T\supset S\},
\allowdisplaybreaks\\
\mathcal A_S&:=\mathcal L^{(1)}_{\supset S}=\{H\in\mathcal A\mid H\supset S\}
.
\intertext{For a non-zero vector $v\in V$ and $S,\,S'\in\mathcal L$, we define}
\langle v,S\rangle&:=\{tv+y\mid t\in\mathbb C,\ y\in S\},
\allowdisplaybreaks\\
\mathcal A_v&:=\{H\in\mathcal A \mid \langle v,H\rangle=V\},
\allowdisplaybreaks\\
\mathcal A_v^c&:=\mathcal A\setminus\mathcal A_v,
\allowdisplaybreaks\\
\mc_v\mathcal A&:=\mathcal A\cup
\{\langle v,S\rangle\mid\codim \langle v,S\rangle=1,\  S\in L(\mathcal A)^{(2)}\}.\hspace{5cm}
\end{align*}

\vspace{-2.5cm}
\hspace{9.3cm}\scalebox{0.8}{\begin{tikzpicture}
\draw
(0,-1)--(0,1.5)
(-0.5,1.5)--(1.5,-0.5)
(1,-1)--(1,1.5);
\draw[very thick,->] (-1.5,0)--(1.5,0);
\node at (1.8,0) {$x_1$};
\draw[dotted, thick] (-1.5,1)--(1.5,1);
\node at (0,1) {$\bullet$};
\node at (0,0) {$\bullet$};
\node at (1,0) {$\bullet$};
\node at (-0.8,0.6) {$v \to$};
\node at (-1.2,-0.3) {$H_1$};
\node at (1.75,-0.7) {$H_4$};
\node at (0,-1.2) {$H_2$};
\node at (1,-1.2) {$H_3$};
\node at (0.2,1.23) {$S$};
\end{tikzpicture}}

We call $\mc_v\mathcal A$ the {\em convolution} of $\mathcal A$ by $v$.
\end{definition}

\begin{remark}\label{rm:funvec}
{\rm (i)} \ 
When we fix a coordinate $x=(x_1,\dots,x_n)$ on $\mathbb C^n$, the $i$-th standard basis
$e_i:=(0,\dots,0,\overset{\overset{\raisebox{-1mm}{\scriptsize i}}\smallsmile}{1},0,\dots,0)$
will occasionally be denoted by $x_i$ for brevity.

\smallskip
{\rm (ii)}\ 
Let $S=H\cap H'\in\mathcal L^{(2)}$ with $H,\,H'\in\mathcal A$.
Suppose
\[
  H=\{x\in\mathbb C^n\mid c_1x_1+\cdots+c_nx_n+c=0\}\text{ and }
  H'=\{x\in\mathbb C^n\mid c_1'x_1+\cdots+c_n'x_n+c'=0\}.
\]
Then $H\in\mathcal A_{x_i}$ if and only if $c_i\ne0$.

If $H,\,H'\in\mathcal A_{x_i}^c$, then $\langle x_i,S\rangle=S$.

If $H\in\mathcal A_{x_i}$ and $H'\in\mathcal A_{x_i}^c$, then $\langle x_i,S\rangle=H'$.

If $H,\,H'\in\mathcal A_{x_i}$, then 
\[\langle x_i,S\rangle=\{x\in\mathbb C^n\mid c_i'(c_1x_1+\cdots+c_nx_n+c)=c_i(c_1'x_1+\cdots+c_n'x_n+c')\}.\]
\end{remark}
%%%%%%%%%%%%%%%%%%%%%%%%%%%%%%%%
%%%%%%%%%%  Def. v-closed %%%%%%
%%%%%%%%%%%%%%%%%%%%%%%%%%%%%%%%
\begin{definition}
A non-zero vector 
$v\in V$ is said to be parallel to $S\in L(\mathcal A)$, or equivalently that 
$S$ is called {\em $v$-closed} if and only if $\langle v,S\rangle = S$.
If  $\langle v,S\rangle \ne S$, $v$ is transversal to $S$. 

$\mathcal A$ is called $v$-closed if and only if 
$\langle v,S\rangle\in L(\mathcal A)$ for every $S\in L(\mathcal A)^{(2)}$.

If there exist $n$ linearly independent vectors $v_1,\dots,v_n
\in \mathbb C^n$ such that $\mathcal A$ is $v_i$-closed for each $i$, 
we say that $\mathcal A$ is {\em stable}.
Equivalently, in the coordinate system defined by $\{v_i\}$
we have
\[\mc_{x_i}\mathcal A=\mathcal A\qquad(i=1,\dots,n).\]
\end{definition}
%%%

\begin{exmp} 
{\rm (i)}\ The braid arrangement 
\[\mathcal A=\bigcup_{1\le i<j\le n}
  \{(x_1,\dots,x_n)\in\mathbb C^n\mid x_i=x_j\}\]
is a stable hyperplane arrangement, which corresponds to mirror hyperplanes 
of the reflection group of type $A_{n-1}$.

\smallskip
{\rm (ii)}\ 
The arrangement of mirror hyperplanes of the reflection group of type $D_n$
\[
 \mathcal A=\{x_i=\pm x_j\mid 1\le i<j\le n\}\quad(n\ge4)
\]
is not stable, which is contained in the stable hyperplane arrangement of 
type $B_n$
\[
  \tilde{\mathcal A}=\mathcal A\cup \{x_i=0\mid 1\le i\le n\}.
\]

\smallskip
{\rm (iii)}\ 
Under the coordinate system $(x,y)$ of $\mathbb C^2$,

$\mathcal A=\bigl\{\{x=y\},\ \{x+y=1\},\ \{x+y=2\}\bigr\}\subset \mathbb C^2$,

$\mathcal L^{(2)}=\bigl\{(\tfrac12,\tfrac12),\ (1,1)\bigr\}$,

$\mc_x\mathcal A=\mathcal A\cup\bigl\{\{y=\tfrac12\},\ \{y=1\}\bigr\}$.

There exists no stable arrangement $\tilde{\mathcal A}$ satisfying 
$\widetilde {\mathcal A}\supset \mathcal A$.

\vspace{-2.3cm}
\hspace{8.4cm}
\scalebox{0.8}{
\begin{tikzpicture}
\draw [thick]
(-0.5,-0.5)--(1.5,1.5)
(-0.5,1.5)--(1.5,-0.5)
(0.5,1.5)--(2.5,-0.5)
;
\draw [->]
(-0.5,0)--(2.5,0);
\draw
(0,-0.5)--(0,1.5)
;
\draw [densely dotted]
(-0.5,1)--(2.5,1)
(-0.5,0.5)--(2.5,0.5)
;
\node at (0.5,0.5) {$\bullet$};
\node at (1,1) {$\bullet$};
\node at (1.5,0.5) {$\circ$};
\node at (0,1) {$\circ$};
\node at (-0.5,-0.7) {$x=y$};
\node at (1.4,-0.7) {$x{+}y{=}1$};
\node at (2.6,-0.7) {$x{+}y{=}2$};
\node at (2.75,0) {$x$};
\node at (0.15,-0.2) {$0$};
\node at (0.95,-0.2) {$1$};
\node at (1.95,-0.2) {$2$};
\end{tikzpicture}}
\end{exmp}

\medskip
The purpose of this paper is to give a classification of stable hyperplane arrangements.
Moreover, in \S\ref{sec:closed}, we determine the vectors $v$ 
for which a stable hyperplane arrangement $\mathcal A$ is $v$-closed. 
We note that the definition of stable arrangements shares features with that 
of fibre-type arrangements % but they are different 
(cf.~\cite{ComplexArrangements}).

Lastly in this introduction, we explain the motivation for this paper.

A Pfaffian system with logarithmic singularities along hypersurface arrangement $\mathcal A$
is given by 
\begin{equation}\label{def:Pf}
  {\mathcal M}: d u = \Omega u,\ \Omega=\sum_{H\in\mathcal A}A_Hd\log f_H\quad(\Omega\wedge\Omega=0)
\end{equation}
where $A_H$ are constant square matrices  of size $N$ and $u$ is a vector of $N$ unknown functions
Each $A_H$ is called the {\em residue matrix} of ${\mathcal M}$  along $H$, 
and the condition $\Omega\wedge \Omega=0$ is the integrability condition of ${\mathcal M}$.
Then the convolution of $\mathcal M$ with respect to the variable $x_i$ and a parameter 
$\mu\in\mathbb C$ produces a new Pfaffian system
\begin{equation}\label{eq:mcPf}
  \widehat{\mathcal M}={\widehat \mc}_{x_i,\mu}{\mathcal M}: d \hat u = \hat \Omega\hat u,\ \hat\Omega
  =\sum_{H\in\mc_{x_i}\!\mathcal A}\hat A_Hd\log f_H.%\quad(\Omega\wedge\Omega=0)
\end{equation}
The middle convolution $\mc_{x_i,\mu}{\mathcal M}$ of $\mathcal M$
is defined (see \cite{Ha}) as an irreducible quotient of $\widehat{\mathcal M}$, 
and the corresponding transformation of solutions is realized 
by a Riemann-Liouville integral.
Here the middle convolution of $\mathcal M$ generalizes the operation
introduced for ordinary differential equations in 
%is defined  as an extension of 
%the middle convolution of ordinary differential equations introduced by 
\cite{katz1996rigid, DR}.
An {\em addition} of $\mathcal M$ is the transformation induced by the gauge change
$u\mapsto \bigl(\prod_{H\in\mathcal A}f_H^{\lambda_H}\bigr)u$
with parameters $\lambda_H\in\mathbb C$. 
A wide class of systems can be generated from a given system by successive additions and 
middle convolutions.
For example, any Fuchsian ordinary differential equation  without an accessory parameter
can be obtained from the trivial equation $u'=0$ (cf.~\cite{katz1996rigid}).
In particular, applying convolution to an addition of the trivial system (with $A_H=0$),
for generic parameters $\lambda_H$, $v$, and $\mu$, yields an irreducible Pfaffian system whose
singular locus equals $\mc_v\mathcal A$.
%A convolution of an addition of the trivial equation $\mathcal M$ with $A_H=0$, which is
%defined by generic $\lambda_H$, $v$ and $\mu$,
%gives an irreducible Pfaffian system with the singular locus $\mc_v\mathcal A$. 
If $\mathcal A$ is stable, these transformations may be analyzed while keeping the singular 
locus fixed.
When $\mathcal A$ is the braid arrangement, the system is of $KZ$-type and the corresponding
transformations were studied by \cite{Ost}.
The non-stable case will be treated in \cite{Ohp}.
%%%%%%%%%%%%%%%%%%%%%%
%%% Classification %%%
%%%%%%%%%%%%%%%%%%%%%%
\section{Classification of stable hyperplane arrangements}

We first examine the property of being ``$v$-closed,'' defined in the previous section, 
which will be used in \cite{Ohp}.

%%%%%%%%%%%%%%%%%%%%%%%%%%%%%%%%%%%%%%%%%
\begin{lemma}\label{lem:arH}
{\rm (i)}\;
Let $S\in\mathcal L$ and $T,\,T'\in\mathcal L^{(\codim S+1)}_{\subset S}$ with $T\ne T'$. 
Then $\mathcal A_T\cap \mathcal A_{T'}=\mathcal A_S$.

{\rm (ii)}\;
The arrangement $\mc_v\mathcal A$ is $v$-closed.

{\rm (iii)}\;
For $S\in\mathcal L$, the subspace $S$ is $v$-closed if and only if
$\mathcal A_S\subset\mathcal A_v^c$.

{\rm (iv)}\;
Let $H\in \mathcal A_v$ and $S\in\mathcal L$. 
If $S$ is $v$-closed, then $H\cap S\ne\emptyset$, and
\begin{equation}
 \mathcal A_{H\cap S}\cap \mathcal A_v^c=\mathcal A_S.
\end{equation}

{\rm (v)}\;
Assume that $\mathcal A$ is $v$-closed and that $S\in\mathcal L$ is not $v$-closed.
Let $H\in\mathcal A_S\cap \mathcal A_v^c$. Then
\begin{equation}
  S = H \cap \langle v,S\rangle 
  \quad\text{and}\quad
  \langle v,S\rangle = \bigcap_{H'\in\mathcal A_S\cap \mathcal A_v^c} H'.
\end{equation}
In particular, the poset $\mathcal L$ is $v$-closed.
\end{lemma}

\begin{proof}
Let $T,\,T'\in\mathcal L^{(\codim S+1)}_{\subset S}$. 
If there exists $H\in (\mathcal A_T\cap \mathcal A_{T'})\setminus\mathcal A_S$,
then $T,\,T'\subset H\cap S \subsetneqq S$, which implies $T=T'=H\cap S$. 
This proves (i).

Let $H,\,H'\in \mc_v\, \mathcal A$ with $\codim\langle v,H\cap H'\rangle =1$.
If $H\notin\mathcal A$, then $\langle v,H\cap H'\rangle=H$, which proves (ii).

Let $H\in\mathcal A$. 
For any $y\in H$, we have $v+y\in H$ if and only if $H\in\mathcal A_v^c$.
Hence $S\in\mathcal L$ is $v$-closed if and only if $\mathcal A_S\subset\mathcal A_v^c$,
which establishes (iii).

Suppose that $H\in\mathcal A_v$ and $S\in\mathcal L$ is $v$-closed.
Since $\langle v,H\rangle=\mathbb C^n$ and $\langle v,S\rangle=S$, 
it follows that $H\cap S\ne\emptyset$. 
Let $H'\in(\mathcal A_{H\cap S}\cap \mathcal A_v^c)\setminus \mathcal A_S$. 
Then $H'\cap S=H\cap S$ and $\mathcal A_{H'\cap S}\subset \mathcal A_v^c$, 
which contradicts the assumption that $H\in\mathcal A_{H\cap S}$. 
This completes the proof of (iv).

Under the assumption of (v), let $S\in\mathcal L^{(k)}$ with $k\ge2$, and suppose
$S=H\cap H_2\cap\cdots\cap H_k$.
Set $H'_j:=\langle v,H\cap H_j\rangle\in\mathcal A_v^c$. Then
$H\cap H_j=H\cap H'_j$ for $j=2,\dots,k$.  
Hence $S=H\cap H'_2\cap\cdots\cap H'_k$ and 
$\langle v,S\rangle=H'_2\cap\cdots\cap H'_k$, which proves (v). 
\end{proof}
%%%%%%%% End of Lemma %%%%%%%%%%

\begin{remark}
Under the projection
\[
  \pi_i: \mathbb C^n\to\mathbb C^{n-1},\quad 
  (x_1,\dots,x_{i-1},x_i,x_{i+1},\dots,x_n)\mapsto 
  (x_1,\dots,x_{i-1},x_{i+1},\dots,x_{n}),
\]
we have 
\[
  L(\mc_{x_i}\mathcal A)
  = L(\mathcal A)\cup\bigcup_{S\in L(\mathcal A)}\pi_i^{-1}\bigl(\pi_i(S)\bigr).
\]
Hence $\mathcal A$ is stable if and only if 
\[
  L(\mathcal A)=\pi_i^{-1}\bigl(\pi_i(L(\mathcal A))\bigr)
  \quad\text{for all }i=1,\dots,n.
\]
\end{remark}

We now study the condition under which $\mathcal A$ is $x_i$-closed. 
Throughout this section, a {\em coordinate transformation} of $(x_1,\dots,x_n)$ 
will usually mean a transformation of the form
\begin{align}\label{eq:Arcoord}
 x_j\mapsto a_jx_{\sigma(j)}+b_j \quad (j=1,\dots,n),
\end{align}
where $a_j,\,b_j\in\mathbb C$, $a_j\ne0$, and $\sigma$ is a permutation of the indices.
Such transformations preserve stability.

\begin{lemma}%[Decomposition]
Suppose that $\mathcal A$ is stable. 
Then there exists a decomposition of the set of indices
\begin{equation}\label{eq:idxdec}
  \{1,\ldots,n\}=\bigsqcup_{j=1}^m I_j
\end{equation}
such that 
\[
\mathcal A_{x_i}\cap \mathcal A_{x_{i'}}\ne \emptyset
\quad\text{if and only if there exists } I_j \text{ with } i,\,i'\in I_j.
\]
\end{lemma}

\begin{proof}
Suppose that 
$\mathcal A_{x_1}\cap \mathcal A_{x_2}\ne\emptyset$,  
$\mathcal A_{x_2}\cap \mathcal A_{x_3}\ne \emptyset$, 
and $\mathcal A_{x_1}\cap \mathcal A_{x_2}\cap \mathcal A_{x_3}=\emptyset$.
Then
\[
  \{x_2=c_1x_1+h_1(x_4,\dots,x_n)\},\quad 
  \{x_2=c_3x_3+h_3(x_4,\dots,x_n)\}\in\mathcal A
\]
for some non-zero constants $c_1,\,c_3$ and suitable functions $h_1$ and $h_3$.
Since $\mathcal A$ is $x_2$-closed, we have
\[
  \mathcal A_{x_1}\cap\mathcal A_{x_3}\ni 
  \{\,c_1x_1+h_1= c_3x_3+h_3\,\}\in\mathcal A.
\]
Thus, if $\mathcal A_{x_i}\cap \mathcal A_{x_j}\ne\emptyset$ and 
$\mathcal A_{x_j}\cap \mathcal A_{x_k}\ne\emptyset$, then 
$\mathcal A_{x_i}\cap \mathcal A_{x_k}\ne\emptyset$. 
This proves the lemma.
\end{proof}

According to the decomposition \eqref{eq:idxdec}, 
the arrangement $\mathcal A$ can be written as
\begin{equation}\label{eq:Ardec}
 \begin{split}
  \mathcal A &= \bigcup_{1\le j\le m}\,\bigcup_{H\in \mathcal A_j}\pi_j^{-1}(H),\\
  \pi_j &:\mathbb C^n\to\mathbb C^{\#I_j},\quad 
  (x_1,\dots,x_n)\mapsto (x_\nu)_{\nu\in I_j},
 \end{split}
\end{equation}
where each $\mathcal A_j$ is a stable hyperplane arrangement in $\mathbb C^{\#I_j}$.
Conversely, for a given decomposition \eqref{eq:idxdec}, 
the arrangement $\mathcal A$ defined by \eqref{eq:Ardec} 
is stable if and only if 
each $\mathcal A_j$ is a stable hyperplane arrangement in $\mathbb C^{\#I_j}$.

To classify stable arrangements, 
we may assume that $\mathcal A$ is {\em indecomposable}; namely,
\begin{align}
   \mathcal A_{x_i}\cap\mathcal A_{x_j}\ne\emptyset
   \qquad(1\le i<j\le n).
\end{align}
We will give representatives of such arrangements 
under suitable coordinate transformations of the form \eqref{eq:Arcoord}.

\medskip
Note that the following lemma is straightforward.

\begin{lemma}\label{lem:Arred}
Let $a,b\in\mathbb C$ with $a\ne0$, and define
\[
  \pi : \mathbb C^n\to\mathbb C^{n-1},\quad
     (x_1,\dots,x_{n-2},x_{n-1},x_n)\mapsto (x_1,\ldots,x_{n-2},a x_{n-1}+b x_n).
\]
Then, for a hyperplane arrangement $\mathcal A'$ in $\mathbb C^{n-1}$, 
the inverse image $\pi^{-1}(\mathcal A')$ 
is a stable hyperplane arrangement in $\mathbb C^n$ 
if and only if $\mathcal A'$ is stable.
\end{lemma}

\begin{definition}
A hyperplane arrangement 
$\mathcal A$ in $\mathbb C^n$ is said to be {\em reducible} 
if $\mathcal A=\pi^{-1}(\mathcal A')$
for some arrangement $\mathcal A'$ in $\mathbb C^{n-1}$ as in Lemma~\ref{lem:Arred}, 
under a suitable coordinate system $(x_1,\dots,x_n)$. 
If $\mathcal A$ is not reducible, 
we say that $\mathcal A$ is {\em reduced}.
\end{definition}

It is clear that an arrangement $\mathcal A$ 
given by the decomposition \eqref{eq:Ardec}
is reduced if and only if 
each $\mathcal A_j$ is reduced.

%%%%%%%%%%%%%%%%%%%%%%%%%%%%%
\begin{remark}[Trivial case]
Assume that $\#\mathcal L^{(2)}=1$.  Let $\mathcal L^{(2)}=\{S\}$. 
Then it follows from the following Lemma~\ref{lem:trivial} that  
\[
\widehat{\mathcal A}
 :=\mc_{x_1}\mc_{x_2}\cdots\mc_{x_n}\mathcal A
  =\mathcal A\cup 
  \{\langle x_i,S\rangle \mid 
     \codim \langle x_i,S\rangle=1,\ \ i=1,\dots,n\}
\]
is stable, and moreover $L^{(2)}(\widehat{\mathcal A})=\{S\}$.
\end{remark}

\begin{lemma}\label{lem:trivial}
Suppose that $\mathcal L^{(2)}=\{S\}$.
Then $S\subset H$ for every $H\in\mc_v\mathcal A$ and every non-zero vector $v\in V$.
\end{lemma}

\begin{proof}
Let $H\in\mathcal A$. Since $\mathcal L^{(2)}\ne\emptyset$, 
there exists $H'\in\mathcal A$ with $H\cap H'\ne\emptyset$, and so
the assumption implies $S=H\cap H'$, and hence
$S\subset\langle v,S\rangle\cap H$.
\end{proof}

\begin{exmp}
The hyperplane arragements
\begin{align*}
 \mathcal A&=\{x_i=\pm1\mid i=1,\dots,4\}\cup\{x_i=x_{i+2}\mid i=1,\,2 \}\subset\mathbb C^4,\\
 \mathcal A'&=\{\{x_1+x_2+x_3=0\},\,\{x_1=\pm1\},\,\{x_2+x_3=\pm1\}\}\subset \mathbb C^3,\\
 \mathcal A''&=\{\{x_1+x_2+x_3=0\},\,\{x_1=x_2\},\,\{2x_1+x_3=0\},\,\{2x_2+x_3=0\}\}\subset \mathbb C^3
\end{align*}
are stable. But $\mathcal A$ is decomposable,  $\mathcal A'$ is reducible and
$\#L^{(2)}(\mathcal A'')=1$. 
\end{exmp}

%%%%%%%%%%%%%%%%%%%%%%%%%%%%%%%%%%%%%%
%%%% Theorem : stable arrangements %%%
%%%%%%%%%%%%%%%%%%%%%%%%%%%%%%%%%%%%%%

We now state the main result of this note.

\begin{theorem}\label{thm:stableAr}
Let $\mathcal A$ be a stable, reduced, and indecomposable hyperplane arrangement in
$\mathbb C^n$.  
Assume that $\#\mathcal L^{(2)}>1$, namely, $\mathcal A$ is non-trivial.
Then, under a suitable coordinate system $(x_1,\ldots,x_n)$, 
there exist a positive integer $m$, a non-negative integer $r$, and non-zero complex numbers 
$\alpha_1,\dots,\alpha_r$ such that $\mathcal A$ has the following form.
Set
\begin{align*}
\Omega     &:=\bigl\{e^{\frac{2\pi k\sqrt{-1}}m}\mid k=1,\dots,m\bigr\},\\
\mathcal A_c&:=\bigl\{ \{x_i=\omega \alpha_j \}\mid 
 			\omega\in \Omega,\ i=1,\dots,n,\ j=1,\dots, r\bigr\},\\
\mathcal A_0&:=\bigl\{\{x_i=0\}\mid 1\le i\le n\bigr\}. 
\end{align*}
If $n=2$, then $r\ge1$ and 
\[
\mathcal A
 =\bigl\{\{x_1=\omega x_2\}\mid \omega\in \Omega'\bigr\}
   \cup\mathcal A_c\cup \mathcal A_0,
   \quad\text{where $1\in\Omega'\subset \Omega$.}
\]
If $n\ge3$, then 
\begin{align*}
\mathcal A&=\mathcal A'
   \quad\text{with $m=1$ and $n>3$, or }\quad
   \mathcal A=\mathcal A'\cup\mathcal A_c\cup \mathcal A_0,\\
\mathcal A'&:=\bigl\{\{x_i=\omega x_j\}\mid \omega\in \Omega ,\ 1\le i<j\le n\bigr\}. 
\end{align*}
\end{theorem}

\begin{remark}
{\rm (i)} 
When $n=2$ in Theorem~\ref{thm:stableAr} and $\Omega'=\{\omega_1,\dots,\omega_L\}$, we may assume that
\[
  \Omega=\{\omega_1^{k_1}\cdots\omega_L^{k_L}\mid k_1,\dots,k_L\in\mathbb Z_{\ge0}\}. 
\]

{\rm (ii)}
Examples of stable arrangements in $\mathbb C^2$ with coordinates $(x,y)$:

\noindent
\begin{tikzpicture}
\draw 
(-1.5,-1.5)--(1.5,1.5)
(-1.5,0) -- (1.5,0)
(0,-1.5)--(0,1.5)
(-0.7,-1.5)--(-0.7,1.5)
(-1.5,-0.7)--(1.5,-0.7)
(0.5,-1.5)--(0.5,1.5)
(-1.5,0.5)--(1.5,0.5)
(1.1,-1.5) -- (1.1,1.5)
(-1.5,1.1)--(1.5,1.1)
;
\node at (2.1,0) {$y=c_i$};
\node at (1.9,1.7) {$x=y$};
\end{tikzpicture}
\ \ 
\raisebox{-3mm}{\begin{tikzpicture}
\draw 
(-1.5,-1.5)--(1.5,1.5)
(-1.5,1.5)--(1.5,-1.5)
(-1.5,0) -- (1.5,0)
(0,-1.5)--(0,1.5)
(-1.5,0)--(1.5,0)
(1,-1.5)--(1,1.5)
(-1.5,1)--(1.5,1)
(-1,-1.5)--(-1,1.5)
(-1.5,-1)--(1.5,-1)

(0.4,-1.5)--(0.4,1.5)
(-1.5,0.4)--(1.5,0.4)
(-0.4,-1.5)--(-0.4,1.5)
(-1.5,-0.4)--(1.5,-0.4)
;
%\node at (2.1,0) {$y=c_i$};
\node at (1.5,1.7) {$x=y$};
\node at (1.5,-1.7) {$x{+}y{=}0$};
\end{tikzpicture}}
\ \ 
\begin{tikzpicture}
\draw 
(-1.5,-1.5)--(1.5,1.5)
(-1.5,0) -- (1.5,0)
(0,-1.5)--(0,1.5)
(-1.5,-0.7)--(1.5,0.7)
(-1.5,0.8)--(1.5,-0.8)
(0.7,-1.5)--(-0.7,1.5)
;
\node at (2.2,0.8) {$y=c_ix$};
\node at (2.1,0) {$y=0$};
\node at (0,1.8) {$x=0$};
\end{tikzpicture}
\end{remark}
%%%%%%%%%%%%%%%%%%%%%%%%
%%% Proof of Theorem %%%
%%%%%%%%%%%%%%%%%%%%%%%%
\section{Proof of Theorem~\ref{thm:stableAr}}
\begin{proof}[Proof of Theorem~\ref{thm:stableAr} for $n=2$]
Choose $(a,b)\in \mathcal L^{(2)}$.
Since $\mathcal A$ is $x_i$-closed, we have $\{x_1=a\},\ \{x_2=b\}\in\mathcal A$.
By the assumption of the theorem, we may therefore assume that, under a suitable coordinate system $(x,y)$ of $\mathbb C^2$,
\[
\{x=0\},\ \{y=0\},\ \{x=1\},\ \{y=1\},\ \{x=y\}\in\mathcal A.
\]
In this case, the condition $\{x=c\}\in\mathcal A$ equals $\{y=c\}\in\mathcal A$.

If $\#(\mathcal A_{x_1}\cap\mathcal A_{x_2})=1$, then the theorem follows immediately with $m=1$.

Hereafter we assume $\#(\mathcal A_{x_1}\cap\mathcal A_{x_2})>1$.  

Now suppose $\{y=x+c\}\in\mathcal A$ for some $c\ne0$.
If $\{x=d\}\in\mathcal A$, then $\{y=c+d\}\in\mathcal A$ and hence $\{x=c+d\}\in\mathcal A$.
By iteration, we obtain $\{x=nc+d\}\in\mathcal A$ for $n=0,1,2,\ldots$, which would imply $\#\mathcal A=\infty$.
Thus, under a suitable coordinate transformation, we may instead assume that $\{y=\alpha_1x\}\in\mathcal A$ with $\alpha_1\ne0,1$.

Hence we assume that
\[
 H_1=\{y=x\},\quad H_2=\{y=\alpha_1x\},\quad 
 H_3=\{y=\alpha_2x+\alpha_3\},\quad H_4=\{x=1\}\in\mathcal A.
\]
Here we do not necessarily assume $H_3\ne H_1$ or $H_3\ne H_2$.

If $\{x=z\}\in\mathcal A$ for some $z\in\mathbb C$, then
$\{x=z\}\cap H_2=\{(z,\alpha_1z)\}\in\mathcal L$, so $\{y=\alpha_1z\}\in\mathcal A$ and
therefore $\{x=\alpha_1z\}\in\mathcal A$.
Similarly, $\{x=z\}\in\mathcal A$ implies $\{x=\alpha_2z+\alpha_3\}\in\mathcal A$.
The desired conclusion then follows directly from Lemma~\ref{lem:Ar20}.
\end{proof}

%%%%%%%%%%%%%%%%%%%%%%%%%%%%%%%
%%%  Lemma for the case n=2 %%%
%%%%%%%%%%%%%%%%%%%%%%%%%%%%%%%

\begin{lemma}\label{lem:Ar20}
Let $\alpha_j\in\mathbb C$ $(j=1,2,3)$ satisfy $\alpha_1\alpha_2(\alpha_1-1)\ne0$, and define
\begin{align*}
 T_1(z)&=\alpha_1z,\\
 T_2(z)&=\alpha_2z +\alpha_3.
\end{align*}
Suppose there exists a finite set $F\subset \mathbb C$ such that
$0\ne z\in F$ and $T_1(F)=T_2(F)=F$.
Then $\alpha_3=0$, and there exists an integer $m\ge2$ such that
\begin{align}\label{eq:n2}
\alpha_1^m=\alpha_2^m=1 
\quad\text{and}\quad
F\supset\bigl\{e^{\frac{2\pi k\sqrt{-1}}m}z\mid k=1,\dots,m\bigr\}.
\end{align}
\end{lemma}

\begin{proof}
Let $a\in F$ be such that $|a|\ge|p|$ for all $p\in F$.
Note that at least one of $a$ or $T_1(a)$ is not a fixed point of $T_2$.  
Since $\#F<\infty$, there exists $m\in\mathbb Z_{>1}$ such that $\alpha_1^m=\alpha_2^m=1$.
For $k=1,2,\ldots,m$, we then have 
$|\alpha_1^k\alpha_2 a+\alpha_3|\le |a|$, hence 
$|a+\alpha_1^{-k}\alpha_2^{-1}\alpha_3|\le |a|$.
Since $\alpha_1\ne1$ and $\alpha_1^m=1$, this inequality holds only when $\alpha_3=0$, which proves \eqref{eq:n2}.
\end{proof}

%%%
We next prepare a lemma that reduces the proof of the theorem to lower-dimensional cases.

\begin{lemma}\label{lem:Arcut}
Let $\mathcal A$ be a hyperplane arrangement in 
$\mathbb C^n$ with coordinates $(x_1,x_2,\dots,x_n)$ and $n\ge3$.
Fix an integer $m$ with $1<m<n$ and a point $p\in\mathbb C^{n-m}$.
Define
\[
V':=\{(x_1,\dots,x_n)\mid (x_{m+1},\dots,x_n)=p\},
\]
and call
\[
\mathcal A':=\{H\cap V'\mid H\in\mathcal A\}\setminus\{\emptyset,\,V'\}
\]
a {\em specialization} of $\mathcal A$. 
For $i=1,\dots,m$, $\mathcal A'$ is $x_i$-closed whenever $\mathcal A$ is $x_i$-closed.
Moreover, if $\mathcal A$ is reduced, then so is $\mathcal A'$. 
\end{lemma}		

\begin{proof}
For $H_1,H_2\in\mathcal A$, set $H_j':=H_j\cap V'$ for $j=1,2$.
Assume $H'_j\ne\emptyset$ and $H'_j\ne V'$ for $j=1,2$, and further that
$H'_1\ne H'_2$ but $H_1'\cap H_2'\ne\emptyset$.
Write
\[
  H_j=\{(x',x'')\in\mathbb C^n\mid f_j(x')=g_j(x'')\},
  \quad
  H_j'=\{x'\in\mathbb C^m\mid f_j(x')=g_j(p)\},
\]
where 
$f_j(x')$ and $g_j(x'')$ are affine linear polynomials in 
$x'=(x_1,\dots,x_m)$ and $x''=(x_{m+1},\dots,x_n)$, respectively, satisfying $f_j(0)=0$.
The assumptions imply that $f_1$ and $f_2$ are linearly independent over $\mathbb C$.
It then follows that 
\[
\langle x_i,H_1'\cap H_2'\rangle
   =\langle x_i,H_1\cap H_2\rangle\cap V'
   \qquad (i=1,\dots,m),
\]
which proves the claim of the lemma.  
Note that $\dim H'_j=m-1$ holds precisely when $f_j\ne0$ and $f_j(0)=0$. 
\end{proof}

\begin{remark}
A stable hyperplane arrangement $\mathcal A$ in Theorem~\ref{thm:stableAr} with $r>0$ and $n\ge2$ 
can be viewed as a suitable specialization of a stable homogeneous hyperplane arrangement. 
For example, by setting $x_3=0$ and $x_4=1$ in 
$\bigl\{\{x_i=x_j\}\subset\mathbb C^4\mid 1\le i<j\le 4\bigr\}$, 
we obtain the hyperplane arrangement 
$\bigl\{\{x_1=x_2\},\,\{x_i=0\},\,\{x_i=1\}\mid i=1,2\bigr\}$ 
in $\mathbb C^2$.
\end{remark}

The following lemma is the key step in proving the theorem for the case $n\ge3$.

\begin{lemma}\label{lem:Ar3n}
Under the assumption of Theorem~\ref{thm:stableAr}, 
we have 
\[
\mathcal A_{x_i}\cap\mathcal A_{x_j}\cap\mathcal A_{x_k}=\emptyset
\quad \text{for } 1\le i<j<k\le n.
\]
\end{lemma}

\begin{proof}[Proof of Theorem~\ref{thm:stableAr} for $n\ge3$ assuming Lemma~\ref{lem:Ar3n}]
We know that 
$\mathcal A_{x_i}\cap\mathcal A_{x_j}\ne\emptyset$
while 
$\mathcal A_{x_i}\cap\mathcal A_{x_j}\cap\mathcal A_{x_k}=\emptyset$ 
for distinct indices $i,j,k\in\{1,\dots,n\}$.
Hence, we may assume that $\{x_i=x_{i+1}\}\in\mathcal L^{(2)}$ for $i=1,\dots,n-1$.
It then follows that $\{x_i=x_j\}\in\mathcal L^{(2)}$ for all $1\le i<j\le n$.
Applying Lemma~\ref{lem:Arcut} together with the case $n=2$, we obtain the theorem.

Indeed, if $\#(\mathcal A_{x_1}\cap\mathcal A_{x_2})=1$, 
then the result corresponds to the case $m=1$.
If there exist two hyperplanes $H_1,H_2\in\mathcal A_{x_1}\cap\mathcal A_{x_2}$, 
we may assume $H_1\cap H_2\subset \{x_1=x_2=0\}$, and the theorem follows immediately.
Note that if $\{x_1=a_2x_2\},\ \{x_2=a_3x_3\}\in\mathcal A$, then 
$\{x_1=a_2a_3x_3\}$ and $\{x_1=a_2a_3x_2\}\in\mathcal A$, and therefore 
$\mathcal A\supset\mathcal A'$ with $m>1$.
\end{proof}

Now we prove Lemma~\ref{lem:Ar3n}, starting with the case $n=3$.

%%%%%%%%%%%%%%%%%%%%%%%%%%%%%%
%%% Lemma for the case n=3 %%%
%%%%%%%%%%%%%%%%%%%%%%%%%%%%%%
\begin{lemma}\label{lem:Ar3}
Let $\mathcal A$ be a stable hyperplane arrangement in $\mathbb C^3$. 
Suppose $\mathcal A_{x_1}\cap\mathcal A_{x_2}\cap\mathcal A_{x_3}\ne\emptyset$. 

{\rm (i)}\,
Then $\bigl\{\{x_1=c_1\},\ \{x_2=c_2\}\bigr\}\not\subset\mathcal A$
for any $c_1,\,c_2\in\mathbb C$.

{\rm (ii)}\,
Moreover, if $\{x_1=c_1\}\subset \mathcal A$ or 
$\#\mathcal L^{(2)}>1$, then $\mathcal A$ is reducible.
\end{lemma}

\begin{proof}
We may assume $H:=\{x_1+x_2+x_3=0\}\in \mathcal A$.
Suppose $\{x_1=c_1\},\{x_2=c_2\}\in\mathcal A$.
Then, by translation, we may further assume $c_1=c_2=0$.
Since $(x_1+x_2+x_3)-x_1 = x_2+x_3$, we have $\{x_2+x_3=0\}\in \mathcal A$.
If $\{nx_1=x_2\}\in\mathcal A$ (which holds for $n=0$), then
the relations
\[
(x_1+x_2+x_3)+(nx_1-x_2)=(n+1)x_1+x_3,\quad
(n+1)x_1+x_3-(x_2+x_3)=(n+1)x_1-x_2,
\]
imply $\{(n+1)x_1=x_2\}\in \mathcal A$, contradicting $\#\mathcal A<\infty$.
Thus we have (i). 

If $\{x_1=c_1\}\subset\mathcal A$, then $\{x_2+x_3=-c_1\}\subset\mathcal A$, 
and (i) implies that $\mathcal A$ is reducible.
In fact, if $H':=\{a_1x_1+a_2x_2+a_3x_3=a_0\}\in \mathcal A$ with $a_2\ne a_3$, 
then $\langle x_1,H\cap H'\rangle=\{(a_2-a_1)x_2+(a_3-a_1)x_3=a_0\}\in\mathcal A$, and moreover
$\{x_2=c_2\}$, $\{x_3=c_3\}\in\mathcal A$ with suitable $c_2$ and $c_3$. 

Now suppose $\#\mathcal L^{(2)}\ge 2$.  
Since $\mathcal A$ is stable, we have $\#\mathcal A_{x_i}^c\ge 1$ for $i=1,2,3$.
Assume $\#\mathcal A_{x_i}^c=1$ for $i=1,2,3$ and write $\mathcal A_{x_i}^c=\{\{H_i\}\}$.
Then for $S\in\mathcal L^{(2)}$, we have $S\subset H_i$ for all $i=1,2,3$.
Since 
$\mathcal A_{x_1}^c\cap\mathcal A_{x_2}^c\cap\mathcal A_{x_3}^c=\emptyset$,  
it follows that $S=H_1\cap H_2\cap H_3$, 
hence $\mathcal L^{(2)}=\{H_1\cap H_2\cap H_3\}$ and $\#\mathcal L^{(2)}=1$.
Thus we may assume $\#\mathcal A_{x_3}^c\ge 2$.

Write 
\[
\mathcal A_{x_3}^c=\bigl\{\{x_1+ax_2=c_i\}\mid i=1,2,\ldots\bigr\},
\]
where $a\ne0$ and $c_1\ne c_2$. 
Since $\{x_1+x_2+x_3=0\}\subset \mathcal A$, 
it follows from Lemma~\ref{lem:Arcut}, Lemma~\ref{lem:Ar3}~(i), 
and the theorem for $n=2$ that $a=1$, and hence $\mathcal A$ is reducible. 
\end{proof}

%%%%%%%%%%%%%%%%%%%%%%%%%%%%%%%%
%%% Lemma for the case n=4 %%%%%
%%%%%%%%%%%%%%%%%%%%%%%%%%%%%%%%
\begin{lemma}\label{lem:Ar4}
Let $\mathcal A$ be a reduced hyperplane arrangement in $\mathbb C^n$
with $n\ge 4$. 
Put $x=x_1,\ y=x_2,\ z=x_3,\ w=x_4,\ x'=(x_5,\dots,x_n).$
If 
\[
\{x+y+z=h_0\},\quad \{w=h_1\},\quad \{y+az+w=h_2\}
\]
belong to $\mathcal A$, then $\mathcal A$ is not stable.
Here $h_j$ are affine linear polynomials in $x'$.
\end{lemma}
\begin{proof}
If $a=0$, then it follows from Lemma~3.5 with the specialization 
$x_4=\dots=x_n=0$ that $\mathcal A$ is not stable. 
If $a=1$, then  $\{x-w=h_0-h_2\}\in\mathcal A$ because $\mathcal A$ is $y$-closed, 
and $\mathcal A$ is not stable as in the case $a=0$.

Thus we may assume $a\ne0,1$.  
Then, from the relations
\small
\begin{align*}
  ax+(a-1)y-w&=a(x+y+z)-(y+az+w)&&(z\text{-closed}),\allowdisplaybreaks\\
  x+(1-a)z-w&=(x+y+z)-(y+az+w)&&(y\text{-closed}),\allowdisplaybreaks\\
  x+(1-a)z&=\bigl(x+(1-a)z-w\bigr)+w&&(w\text{-closed}),\allowdisplaybreaks\\
  ax+(a-1)y&=\bigl(ax+(a-1)y-w\bigr)+w&&(w\text{-closed}),\allowdisplaybreaks\\
  (a-1)y+(a^2-a)z-w&=(ax+(a-1)y-w)-a(x+(1-a)z)&&(x\text{-closed}),\allowdisplaybreaks\\
  ax+(a-a^2)z+w&=\bigl(ax+(a-1)y\bigr)-\bigl((a-1)y+(a^2-a)z-w\bigr)&&(y\text{-closed}),\allowdisplaybreaks\\
  2ax+(a-1)y+(a-a^2)z&=\bigl(ax+(a-a^2)z+w\bigr)+\bigl(ax+(a-1)y-w\bigr)&&(w\text{-closed}),\allowdisplaybreaks\\
  ax+(a-a^2)z+(1-a)w&=(1-a)(y+az+w)+\bigl(ax+(a-1)y\bigr)&&(y\text{-closed}),\allowdisplaybreaks\\
  (1-a)(y+az+2w)&=2\bigl(ax+(a-a^2)z+(1-a)w\bigr)\\
  &\quad{}-\bigl(2ax+(a-1)y+(a-a^2)z\bigr)&&(x\text{-closed}),
\end{align*}
\normalsize
we obtain $\{y+az+2w=h_3\}\in\mathcal A$.
Hence $\{y+az+2^n w=h_{n+2}\}\in\mathcal A$ for $n=0,1,2,\ldots$,
where each $h_{n+2}(x')$ is an affine linear polynomial in $x'$.
\end{proof}

\begin{proof}[Proof of Lemma~\ref{lem:Ar3n}] 
Let $m$ be the maximal integer such that there exist indices 
$i_\nu$ satisfying $1\le i_1<\cdots<i_m\le n$ and 
$\mathcal A_{x_{i_1}}\cap\cdots\cap\mathcal A_{x_{i_m}}\ne\emptyset$.
We will show that $m=2$.

Assume to the contrary that $m>2$. Without loss of generality, we may assume
$H:=\{x_1+\cdots+x_m=0\}\in\mathcal A$.

Suppose $m<n$. Choose $H_2\in\mathcal A_{x_2}\cap \mathcal A_{x_{m+1}}$.  
By the maximality of $m$, there exists $i$ with $1\le i\le m$ such that 
$H_2\in\mathcal A_{x_i}^c$. 
We may assume $i=1$, and write
\[
  H_2=\{x_2+a_3x_3+\cdots+a_mx_m+x_{m+1}+\cdots + a_nx_n=a_0\}.
\]
Suppose $\mathcal A_{x_1}^c\cap\cdots\cap\mathcal A_{x_m}^c\ne\emptyset$,
and let $H_3\in\mathcal A_{x_1}^c\cap\cdots\cap\mathcal A_{x_m}^c$. 
We may assume $H_3\in\mathcal A_{x_{m+1}}$.
By setting $x_4=\cdots=x_m=0$, 
Lemmas~\ref{lem:Arcut} and~\ref{lem:Ar4} imply that 
$\mathcal A$ is not stable, 
hence $\mathcal A_{x_1}^c\cap\cdots\cap\mathcal A_{x_m}^c=\emptyset$. 
\[
  (\overset{x}{x_1},\overset{y}{x_2},\overset{z}{x_3},\dots,x_m,\overset{w}{x_{m+1}},\dots,x_n)
\]

Thus, including the case $m=n$, we may assume
\[
 H_2=\{x_2+a_3x_3+\cdots+a_nx_n=a_0\}\in\mathcal A_{x_1}^c.
\]
Since $\#\mathcal L^{(2)}>1$, there exists 
\[
H_3=\{b_1x_1+b_2x_2+b_3x_3+\cdots+b_nx_n=b_0\}\in\mathcal A
\]
such that $H_3\not\supset H_1\cap H_2$.

If $H_3\cap H_1=\emptyset$, then by Lemma~\ref{lem:Ar3n} 
and Lemma~\ref{lem:Arcut} (with the specialization $x_4=\cdots=x_n=0$),
$\mathcal A$ is either reducible or not stable.  
Hence we must have $H_3\cap H_1\ne\emptyset$. 

Since $\mathcal A$ is $x_1$-closed, we may assume $b_1=0$.
Similarly, as $H_3\cap H_2\ne\emptyset$, 
the $x_2$-closedness of $\mathcal A$ implies that we may also assume $b_2=0$.
Because $\mathcal A$ is non-trivial, there exists some $b_i\ne0$ with $3\le i\le m$;
we may take $b_3\ne0$.
Then, applying Lemma~\ref{lem:Ar3} to the restriction $x_4=\cdots=x_n=0$, 
we conclude that $\mathcal A$ is not stable.
\end{proof}
%%%%%%%%%%%%%%%%%%%%%%
%%% Related result %%%
%%%%%%%%%%%%%%%%%%%%%%
\section{A related result}\label{sec:closed}
In this final section, we determine all vectors $v$ for which a stable hyperplane arrangement 
$\mathcal A$ is $v$-closed.

\begin{proposition}
Let $\mathcal A$ be the hyperplane arrangement described in Theorem~\ref{thm:stableAr},
and assume that $\#\mathcal L^{(2)}>1$.
Let $v$ be a non-zero vector in $\mathbb C^n$ such that $\mathcal A$ is $v$-closed.

If $r>0$ or $m>1$, then $v$ is a scalar multiple of one of the coordinate vectors $x_i$.

If $m=1$ and $\mathcal A=\mathcal A'\cup \mathcal A_0$, 
then $v$ is a scalar multiple of either one of the $x_i$ 
or of $(c,\dots,c)\in\mathbb C^n$ with $c\ne0$.

If $m=1$ and $\mathcal A=\mathcal A'$, 
then $v$ is a scalar multiple of one of the $x_i$ modulo $\mathbb C(1,\dots,1)$.
\end{proposition}

\begin{proof}
We may assume $v=(1,c_2,c_3,\ldots)$ with $c_2\ne0$. 

Suppose $H_1=\{x_1=0\}$ and $H_2=\{x_2=\alpha_1\}$ are in $\mathcal A$.
Then
\[
  \langle v,H_1\cap H_2\rangle=\{x_2=c_2x_1+\alpha_1\}\in\mathcal A.
\]
Hence $r=0$ and $n>2$. 
Since $H_3=\{x_2=\omega x_3\}\in\mathcal A$, we also have
\[
  \langle v,H_1\cap H_3\rangle
  =\{x_2-\omega x_3=(c_2-\omega c_3)x_1\}\in\mathcal A,
\]
which implies $c_2=\omega c_3$, and therefore $m=1$ and $c_2=c_3$.
By symmetry of the coordinates (cf.~$\sigma$ in \eqref{eq:Arcoord}), 
we conclude that $m=1$ and $v=(1,\dots,1)$.

Thus we may assume 
\[
  \mathcal A=\bigl\{\{x_i=x_j\}\mid 1\le i<j\le n\bigr\}
  \quad\text{with } n\ge4.
\]
We may further assume $v=(1,0,c_3,c_4,\dots)$.
Then
\begin{align*}
 \langle v,\{x_1-x_2=x_3-x_4 \}\rangle
   &=\{x_2-x_3=(c_3-c_4)(x_1-x_2)\}\in\mathcal A,\\
 \langle v,\{x_1-x_2=x_2-x_3\}\rangle
   &=\{x_2-x_3=-c_3(x_1-x_2)\}\in\mathcal A.
\end{align*}
Hence $c_3=c_4=0$ or $1$.
By symmetry, it follows that 
$v=(1,0,0,\dots,0)$ or $(1,0,1,\dots,1)$,
which proves the claim.
\end{proof}

\begin{remark} {\rm (i)} 
For $m=1$, the expression of the arrangement 
$\mathcal A=\mathcal A'\cup\mathcal A_0$ 
in Theorem~\ref{thm:stableAr} remains the same under the coordinate system 
\[
  (x_1-x_n,\ldots,x_{n-1}-x_n,-x_n)
\]
of $\mathbb C^n$.

{\rm (ii)}
The arrangement $\mathcal A=\{x_i\pm x_j=0\mid 1\le i<j\le 3\}$ is 
$(\epsilon_1,\epsilon_2,\epsilon_3)$-closed  
when $\epsilon_k\in\{1,-1\}$ for $k=1,2,3$.
\end{remark}
%%%%%%%%%%%%%

\end{document}